\documentclass[11pt, reqno]{amsart}
\usepackage{amsthm}
\usepackage{amsmath,amssymb,amsthm,mathrsfs,amscd}
\usepackage{MnSymbol}
\usepackage{color}
\usepackage[all]{xy}
\usepackage[english]{babel}
\usepackage{graphicx}
\usepackage[latin1]{inputenc}
\usepackage[T1]{fontenc}
\usepackage[final=true]{hyperref}
\usepackage{tikz}
\usetikzlibrary{patterns}
\usetikzlibrary{arrows}
\usepackage{rotating}
\usepackage{cases}
\usepackage{mathtools}
\theoremstyle{plain}
\newtheorem{thm}{Theorem}[section]

\newtheorem{prop}[thm]{Proposition}

\theoremstyle{definition}
\newtheorem{ex}[thm]{Example}
\newtheorem{rem}[thm]{Remark}
\theoremstyle{definition}
\newtheorem{defi}[thm]{Definition}

\DeclareMathOperator{\Z}{\mathbb{Z}}

\DeclareMathOperator{\sgn}{sgn}
\DeclareMathOperator{\wt}{wt}
\DeclareMathOperator{\opp}{opp}

\DeclareMathOperator{\ii}{\mathbf{i}}
\DeclareMathOperator{\jj}{\mathbf{j}}

\DeclareMathOperator{\W}{\mathcal{W}(w_0)}

\newcommand\restr[2]{{
  \left.\kern-\nulldelimiterspace 
  #1 
  \vphantom{\big|} 
  \right|_{#2} 
  }}
  \makeatletter
\renewcommand*\env@cases[1][1.2]{%
  \let\@ifnextchar\new@ifnextchar
  \left\lbrace
  \def\arraystretch{#1}%
  \array{@{}l@{\quad}l@{}}%
}
\makeatother

\begin{document}

\title[Interplay of Lusztig and string data]{On the interplay of the parametrizations of canonical bases by Lusztig and string data}

\author{Volker Genz}
\address{Department of Mathematics, Ruhr-University Bochum}
\email{volker.genz@gmail.com}

\author{Gleb Koshevoy}
\address{IITP RAS, Moscow}
\email{koshevoyga@gmail.com}

\author{Bea Schumann}
\address{Mathematical Institute, University of Cologne}
\email{bschumann@math.uni-koeln.de}

\begin{abstract} For arbitrary reduced words we give formulas for the crystal structures on string and Lusztig data of type $A$ as well as the defining inequalities of the corresponding polytopes revealing certain dualities between them.
\end{abstract}

\maketitle

\section{Introduction}

The parametrizations of Lusztig's canonical bases of irreducible finite dimensional representations of a simple complex Lie algebra $\mathfrak{g}$ by Lusztig- and string data, respectively, are intensely studied objects with numerous applications. They have proven useful in the construction of canonical bases \cite{L93,K95} toric degeneration of flag varieties \cite{FFL, FN}, geometric crystals \cite{BK07} and the study of the total positive part of the corresponding algebraic group leading to the notion of a cluster algebra \cite{BFZ}. 

The parametrizations come in families indexed by reduced words $\ii$ of the longest Weyl group element of $\mathfrak{g}$. Each parametrization is given by integer points of a convex polytope which is cut out of a polyhedral convex cone whose integer points parametrize Lusztig's canonical basis of the negative part $U_q^-$ of the quantized enveloping algebra $U_q(\mathfrak{g})$. 

By construction, the integer points of these cones carry a crystal structure isomorphic to $B(\infty)$, the crystal base of $U_q^-$. For a reduced word $\ii$ we denote the crystal given by Lusztig data by $\mathcal{L}_{\ii}(\infty)$ and the crystal given by string data by $\mathcal{S}_{\ii}(\infty)$. We denote the corresponding $*$ crystal structure by $\mathcal{L}_{\ii}(\infty)^*$ and $\mathcal{S}_{\ii}(\infty)^*$, respectively. Each of these crystals leads to the description of a crystal isomorphic to $B(\lambda)$, the crystal base of the irreducible representation $V(\lambda)$, with underlying sets ($I$ the index set of fundamental weights and $\lambda_a=\left<\lambda,\omega_a\right>$) 
\begin{align*}
\displaystyle{\mathcal{S}_{\ii}(\lambda)}&=\displaystyle{\{ x \in \mathcal{S}_{\ii}(\infty) \mid \varepsilon_a^*(x) \le \lambda_a \ \forall a\in I\}}, 
\,\displaystyle{\mathcal{S}_{\ii}(\lambda)^*}=\displaystyle{\{ x \in \mathcal{S}_{\ii}(\infty) \mid \varepsilon_a(x) \le \lambda_a \ \forall a\in I\}}, \\
\displaystyle{\mathcal{L}_{\ii}(\lambda)}&=\displaystyle{\{ x \in \mathcal{L}_{\ii}(\infty) \mid \varepsilon_a^*(x) \le \lambda_a \ \forall a\in I\}}, 
\,\displaystyle{\mathcal{L}_{\ii}(\lambda)^*}=\displaystyle{\{ x \in \mathcal{L}_{\ii}(\infty) \mid \varepsilon_a(x) \le \lambda_a \ \forall a\in I\}}.
\end{align*}
We refer to Section \ref{crysdef} for the definitions of $\varepsilon_a$ and $\varepsilon^*_a$. 
Each of the spaces $\mathcal{S}_{\ii}(\lambda)^*$, $\mathcal{S}_{\ii}(\lambda)$, $\mathcal{L}_{\ii}(\lambda)$ and $\mathcal{L}_{\ii}(\lambda)^*$ is given by the integer solution of a system of linear inequalities. The polytopes given by the real valued solutions of the system of linear inequalities defining $\mathcal{S}_{\ii}(\lambda)^*$ and $\mathcal{S}_{\ii}(\lambda)$ are referred to as the (Littelmann--Berenstein-Zelevinsky) string polytope and the Nakashima-Zelevinsky (string) polytope, respectively. 

The crystal structures allow to read off characters, tensor product multiplicities and branching multiplicities for the restriction to Levi subalgebras. Moreover, each crystal comes with natural operators by which one can successively generate all integer points of the polytopes starting from the integer point corresponding to the highest weight vector.

In this work we give an explicit description of each of the crystals $\mathcal{S}_{\ii}(\lambda)$, $\mathcal{S}_{\ii}(\lambda)^*$, $\mathcal{L}_{\ii}(\lambda)$ and $\mathcal{L}_{\ii}(\lambda)^*$ for $\mathfrak{g}=\text{sl}_n(\mathbb{C})$ consisting of

\begin{enumerate}
\item an order lattice $(\mathbb{L}_a, \preceq)$ of sequences of positive roots for any $a\in [n-1]$, 
\item two maps associating vectors $r(\ell), s(\ell)$ to any $\ell \in \mathbb{L}_a$.
\end{enumerate}

We define three such order lattices $(\Gamma_a,\preceq), (\Gamma^{*}_a,\preceq), (\Upsilon_a,\preceq)$ with corresponding maps in Section \ref{orderlattices}. The crystal structures are then given in Theorem \ref{crossingformulas}.

Using the same ingredients we derive in Section \ref{conesandpol} inequalities of the corresponding cones and polytopes.

Note that the defining inequalities of the polytopes naturally divide in two subclasses: the ones defining the overlying cone and the ones cutting out the polytope from the cone. We call the former the cone inequalities and the latter the highest weight inequalities.

The following table summarizes the role the order lattices $(\Gamma_a,\preceq)$, $(\Gamma^{*}_a,\preceq)$, $(\Upsilon_a,\preceq)$ play in the description of crystal structures and defining inequalities.
\begin{center}
 \resizebox{\textwidth}{!}{
\begin{tabular}{ c|c|c|c }

Polytopes $\backslash$ lattices & $\Gamma_a$ & $\Gamma^{*}_a$ & $\Upsilon_a$ \\ \hline
$\mathcal{L}_{\ii}(\lambda)$ & crystal description & highest weight inequalities & cone inequalities \\ \hline
$\mathcal{L}_{\ii}(\lambda)^*$ & highest weight inequalities & crystal description & cone inequalities \\ \hline
$\mathcal{S}_{\ii}(\lambda)$ & highest weight inequalities & cone inequalities & crystal description \\ \hline
$\mathcal{S}_{\ii}(\lambda)^*$ & crystal description & cone inequalities & highest weight inequalities
\end{tabular}}
\end{center}

\section*{Acknowledgement}
B. Schumann would like to thank Peter Littelmann for inspiring conversations and constant encouragement. She is further grateful for interesting discussions with Shmuel Zelikson. V. Genz and B. Schumann were (partially) supported by the SFB/TRR 191. G. Koshevoy was supported by the grant RSF 16-11- 10075.

\section{The parametrizations and their crystals structures}

\subsection{Notation}
Let $\mathbb{N}=\{0,1,2,\ldots\}$, $\mathfrak{g}=\text{sl}_n(\mathbb{C})$ and $\mathfrak{h}\subset \mathfrak{g}$ its Cartan subalgebra consisting of the diagonal matrices. We set
$[m]:=\{1,2,\ldots, m\}$
 and define for $k\in [n]$ the function $\epsilon_k\in \mathfrak{h}^*$ by $\epsilon_k(\text{diag}(h_1,h_2,\ldots,h_n))=h_k$. The set $\Phi^+$ of positive roots of $\mathfrak{g}$ is given by $\Phi^+=\{\epsilon_k-\epsilon_{\ell} \mid 1\leq k<\ell \leq n\}$. We denote by $N:=\frac{n(n-1)}{2}$ the cardinality of $\Phi^+$ and write $\alpha_{k,\ell}=\epsilon_k-\epsilon_{\ell}$. 
We identify the set $\Phi^+$ with the set of pairs $(i,j)\in [n]^2$, $i<j$ by
 \begin{equation}\label{eq:rootident}
 \alpha_{k, \ell} \mapsto (k,\ell).
 \end{equation}

The fundamental weights of $\mathfrak{g}$ are given as $\omega_a=\sum_{s\in [a]}\epsilon_s$ for $a\in [n-1]$. Let $P\subset \mathfrak{h}^*$ (resp. $P^+\subset \mathfrak{h}^*$) the $\mathbb{Z}$-span (resp. $\mathbb{Z}_{\ge 0}$-span) of the set of fundamental weights $\{\omega_a\}_{a\in [n-1]}$ of $\text{sl}_n(\mathbb{C})$. The lattice $P$ is called the weight lattice of $\mathfrak{g}$ and $P^+$ the set of dominant integral weights. For $\lambda \in P^+$ we denote by $V(\lambda)$ the irreducible $\mathfrak{g}$-representation of highest weight $\lambda$ and write $\lambda=\sum_{a\in [n-1]}\lambda_a\omega_a$.

\subsection{Background on crystals}\label{crysdef}
We refer to \cite{K95} for an introduction to crystals and only give the necessary ingredients for this work.

A crystal consists of a set $B$ together with maps $\wt: B \rightarrow P$, ${e}_a:B \rightarrow B \sqcup \{0\}$ and ${f}_a:B  \rightarrow B \sqcup \{0\}$ for $a\in [n-1]$ satisfying $b'={e}_a b \iff {f}_a b'=b$ for $b,b' \in B$.

We denote by $B(\infty)$ the crystal base of negative part of the quantized enveloping algebra $U^-_q(g)$ of $\mathfrak{g}$ (see Section 8 of \cite{K95}) and by $B(\lambda)$ the crystal base of $V(\lambda)$ (see Section 4 of \cite{K95}). On these two crystals we have a function $\varepsilon_a(b)=\max\{m\in \mathbb{N} \mid {e}_a^{m}b \ne 0\}$ for any $a\in [n-1]$. On $B(\lambda)$ we further have a function $\varphi_a(b)=\max\{m\in \mathbb{N} \mid {f}_a^{m}b \ne 0\}=\varepsilon_a(b)+\left<\wt(b),\alpha_{a,a+1}\right>$.

Kashiwara \cite[Section 8.3]{K95} introduced a $\wt$-preserving involution on $B(\infty)$ and we denote by $f_a^*(x)=(f_ax^*)^*, e_a^*(x)=(e_ax^*)^*$ and $\varepsilon_a^*(x)=\varepsilon_a(x^*)$ the twisted maps. Further $B(\infty)^*$ is the crystal given by the set $B(\infty)$ and the twisted maps.

By \cite[Section 8.3]{K95} the crystal $B(\lambda)$ can be identified with the set
\begin{equation}\label{lambdaininfty}
B(\lambda)=\{b\in B(\infty) \mid \varepsilon^*_a(b) \le \lambda_a \ \forall a \in [n-1]\}
\end{equation}
with maps ${e}_a$, $\varepsilon_a$ induced from those on $B(\infty)$, $\wt$ the induced map $\wt$ on $B(\infty)$ shifted by $\lambda$ and ${f}_a b$ the induced map from $B(\infty)$ if $\varphi(b)>0$ and otherwise ${f}_a b=0$.

Exchanging the roles of $B(\infty)$ and $B(\infty)^*$, we further identify $B(\lambda)$ with 
\begin{equation}\label{lambdaininftystar}
B(\lambda)=\{b\in B(\infty)^* \mid \varepsilon_a(b) \le \lambda_a \ \forall a \in [n-1]\}
\end{equation}
and with maps as above twisted by $*$.

\subsection{Symmetric groups and reduced words}
Let $\mathfrak{S}_n$ be the symmetric group in $n$ letters. The group $\mathfrak{S}_n$ is generated by the simple transpositions $\sigma_a$, $a \in [n-1]$, interchanging $a$ and $a+1$.

A reduced expression of $w\in \mathfrak{S}_n$ is a decomposition of $w=\sigma_{i_1}\sigma_{i_2}\cdots \sigma_{i_k}$
into a product of simple transposition with a minimal number of factors. We call $k$ the \emph{length $\ell(w)$ of $w$}. The group $\mathfrak{S}_n$ has a unique longest element $w_0$ of length $N:=\frac{n(n-1)}{2}$.

For a reduced expression of $w_0$ we write $\ii:=(i_1,i_2,\ldots, i_N)$ and call $\ii$ a \emph{reduced word} (for $w_0$). The set of reduced words for $w_0$ is denoted by $\W$. Every $\ii=(i_1,i_2,\ldots, i_N) \in \W$ induces a total ordering $\le_{\ii}$ on $\Phi^+$ given by
\begin{equation}\label{eq:rootorder}
(i_1, i_1+1) \le_{\ii} \sigma_{i_1} (i_2, i_2+1) \le_{\ii} \ldots \le_{\ii} \sigma_{i_1}\cdots\sigma_{i_{N-1}}(i_N,i_N+1),
\end{equation}
where we used the identification \eqref{eq:rootident}.

\subsection{Crystal structures on string data}\label{stringcrys}
Let $\ii \in \mathcal{W}(w_0)$.  Using the definition of $\varepsilon^*_a$ and ${e}^*_a$ from Section \ref{crysdef} we associate to $b\in B(\infty)$ a vector $\text{str}_{\ii}(b)\in \mathbb{N}^N$ as follows.
Let $x_k=\varepsilon^*_{i_k}(({{e}^{*}}_{i_{k-1}})^{x_{k-1}}\cdots ({{e}^*}_{i_1})^{x_1}b)$. 
We define $\text{str}_{\ii}(b):=(x_1,\ldots,x_N)$ and call $(x_1,\ldots,x_N)$ the string datum of $b$ in direction $\ii$.

By \cite{BZ, Lit} the cone $\mathcal{S}_{\ii}\subset \mathbb{R}^N$ spanned by the set $\{\text{str}_{\ii}(b) \mid b\in B(\infty)\}$ is a rational polyhedral cone called the \emph{string cone} associated to $\ii$.

\begin{rem} We remark that in \cite{BZ, Lit, GKS1} the string cone $\mathcal{S}_{\ii}\subset \mathbb{R}^N$ is defined as the conic hull of the points $\{\text{str}_{\ii}(b^*) \mid b\in B(\infty)\}$.
Since $B(\infty)=B(\infty)^*$, both descriptions of the string cone coincide.
\end{rem}

Identifying the string cone $\mathcal{S}_{\ii}$ with the image of $B(\infty)$ using Kashiwara's embedding theorem \cite[Theorem 8.2]{K95} we obtain a crystal structure on $\mathcal{S}_{\ii}$ isomorphic to $B(\infty)$.
We refer to Section 2.4 of \cite{NZ} and \cite{GKS3} for a detailed discussion.

We recall the crystal structure on $\mathcal{S}_{\ii}$.

Let $(c_{i,j})$ denote the Cartan matrix of $\text{sl}_n(\mathbb{C})$. For $x\in\mathcal{S}_{\ii}$, $a\in[n-1]$ and $k\in[N]$ 
\begin{align*}\begin{split}
\eta_k(x)&:=x_j+ \sum_{ j <k\le N}c_{i_j,i_k}x_k,\\
\varepsilon_a(x)&= \max\left\{\eta_k(x) \mid k\in[N],\, i_k=a \right\},\qquad \wt(x)=-\displaystyle\sum _{k=1}^{N} x_k \alpha_{i_k},\\
{f}_a (x) & = x+\left(\delta_{k,\ell^{x}}\right)_{k\in[N]}\\ {e}_a (x) &=\begin{cases}x-\left(\delta_{k, \ell_{x}}\right)_{k\in[N]} & \text{if }\varepsilon_a(x)>0, \\ 0, &\text{else,}\end{cases}\end{split}
\end{align*}
where $\ell^{x}\in[N]$ is minimal with $i_{\ell^x}=a$ and $\eta_{\ell^x}(x)=\varepsilon_a(x)$ and where
$\ell_{x}\in[N]$ is maximal with $i_{\ell_x}=a$ and $\eta_{\ell_x}(x)=\varepsilon_a(x)$. We denote the set $\mathcal{S}_{\ii}$ equipped with this crystal structure by $\mathcal{S}_{\ii}(\infty)$.

The $*$-crystal structure on $\mathcal{S}_{\ii}$ has the following description. By \cite{BZ, Lit} there exist piecewise linear bijections $\Psi^{\ii}_{\jj} : \mathcal{S}_{\ii} \rightarrow \mathcal{S}_{\jj}$ with $\text{str}_{\jj}=\Psi^{\ii}_{\jj} \circ \text{str}_{\ii}.$

Let $x \in \mathcal{S}_{\ii},$ $a\in[n-1]$ and $\jj\in\W$ with $j_1=a$. Setting $y:=\Psi^{\ii}_{\jj}(x)\in\mathcal{S}_{\jj}$ we have
\begin{align*}
\varepsilon_{a}^*(x)&=y_1,\qquad \wt(x)=-\displaystyle\sum _{k=1}^{N} x_k \alpha_{i_k},\\\notag
f_{a}^*(x) &= \Psi^{\jj}_{\ii}\left(y+(1,0,0,\dots)\right),\\\notag
e_{a}^*(x) &= \begin{cases}
\Psi^{\jj}_{\ii}\left(y-(1,0,0,\dots) \right), & \text{if $\varepsilon_a^*(x)>0,$}\\0,&\text{else.}
\end{cases}
\end{align*}	
We denote the set $\mathcal{S}_{\ii}$ equipped with this crystal structure by $\mathcal{S}_{\ii}(\infty)^*$.

By \eqref{lambdaininfty}, \eqref{lambdaininftystar} we get a crystal structures isomorphic to $B(\lambda)$ on the sets 
\begin{align*}
\mathcal{S}_{\ii}(\lambda)^*&=\{x \in \mathcal{S}_{\ii} \mid \varepsilon_a(x) \le \lambda_a \ \forall a\in [n-1]\}\subset \mathbb{N}^N,\\
\mathcal{S}_{\ii}(\lambda)&=\{x \in \mathcal{S}_{\ii} \mid \varepsilon^*_a(x) \le \lambda_a \ \forall a\in [n-1]\}\subset \mathbb{N}^N,
\end{align*}	
induced from $\mathcal{S}_{\ii}(\infty)^*$ and $\mathcal{S}_{\ii}(\infty)$, respectively.

The sets $\mathcal{S}_{\ii}(\lambda)^*$ and $\mathcal{S}_{\ii}(\lambda)$ are given by the integer solution of a system of linear inequalities. The polytope given by the real valued solutions of the system of linear inequalities defining $\mathcal{S}_{\ii}(\lambda)^*$ is referred to as the (Littelmann--Berenstein-Zelevinsky) string polytope (see \cite{BZ,Lit,FN}). The polytope given by the real valued solutions of the system of linear inequalities defining $\mathcal{S}_{\ii}(\lambda)$ is referred to as the Nakashima-Zelevinsky (string) polytope (see \cite{FN, NZ, N99})

\subsection{Crystal structures on Lusztig data}\label{sec:crystal}
By identifying $x\in \mathbb{N}^N$ with the divided powers of the root vectors in a monomial of a PBW-type $B_{\ii}$ basis associated to $\ii=(i_1,i_2,\ldots,i_N)\in \mathcal{W}(w_0)$ we obtain a weight preserving bijection of $\mathbb{N}^N$ and Lusztig's canonical basis \cite{L93}. We call $x\in \mathbb{N}^N$ the $\ii$-Lusztig datum of the corresponding canonical basis element.

Let $\ii$ and $\jj$ be two reduced words for $w_0$. In \cite{L93} a piecewise-linear bijection $\Phi_{\jj}^{\ii}:\mathbb{N}^{N}\rightarrow \mathbb{N}^{N}$ from the set of $\ii$-Lusztig data to the set of $\jj$-Lusztig data is given. The piecewise-linear bijections $\Phi_{\jj}^{\ii}$ are used to equip the set $\mathbb{N}^{N}$ with crystal structures isomorphic to $B(\infty)$ and $B(\infty)^*$ (\cite{L93}, see also \cite{GKS1}) as follows.

Let $x\in\mathbb{N}^N$ be an $\ii$-Lusztig datum and $a\in[n]$. We set $\wt(x)=\sum_{k\in {N}}x_k\beta_k$ where $\Phi^+_{\ii}=(\beta_1,\beta_2,\ldots,\beta_N)$ is the sequence of positive roots ordered with respect to $\le_{\ii}$. Let $\jj\in\W$ with $j_1=a$ and $y:=\Phi^{\ii}_{\jj}(x)$
\begin{align*}
\varepsilon_{a}(x)&=y_1,\\
f_{a}(x) &= \Phi^{\jj}_{\ii}\left(y+(1,0,0,\dots)\right),\\
e_{a}(x) &= \begin{cases}
\Phi^{\jj}_{\ii}\left(y-(1,0,0,\dots) \right), & \text{if $\varepsilon_a(x)>0,$}\\0,&\text{else.}
\end{cases}
\end{align*}
Let $\jj\in\W$ with $j_N=n-a$ and $y:=\Phi^{\ii}_{\jj}(x)$
\begin{align*}
\varepsilon_{a}^*(x)&=y_N,\\
f_{a}^*(x) &= \Phi^{\jj}_{\ii}\left(y+(0,\dots,0,1)\right),\\
e_{a}^*(x) &= \begin{cases}
\Phi^{\jj}_{\ii}\left(y-(0,\dots,0,1) \right), & \text{if $\varepsilon_a^*(x)>0,$}\\0,&\text{else.}
\end{cases}
\end{align*}

By \eqref{lambdaininfty}, \eqref{lambdaininftystar} we get a crystal structures isomorphic to $B(\lambda)$ 
\begin{align*}
\mathcal{L}_{\ii}(\lambda)^*&=\{x \in \mathcal{L}_{\ii} \mid \forall a\in [n-1] \ \varepsilon_a(x) \le \lambda_a\}\subset \mathbb{N}^N,\\
\mathcal{L}_{\ii}(\lambda)&=\{x \in \mathcal{L}_{\ii} \mid  \forall a\in [n-1] \ \varepsilon^*_a(x) \le \lambda_a \}\subset \mathbb{N}^N
\end{align*}
induced from $\mathcal{L}_{\ii}(\infty)^*$ and $\mathcal{L}_{\ii}$, respectively.

\section{Order lattices associated to crystal structures}\label{orderlattices}

\subsection{Order lattice of (dual) Reineke crossings}

In this section we introduce the order lattices $\Gamma_a$ and $\Gamma^{*}_a$ for $a\in I$. The main combinatorial object we use is the \emph{wiring diagram} $\mathcal{D}_{\ii}$ associated to $\ii=(i_1,i_2,\ldots i_N)\in \W$. The diagram $\mathcal{D}_{\ii}$ consists of a family of $n$ piecewise-straight lines, called \emph{wires} with labels in the set $[n]$. Each vertex $(p,q)$ of $\mathcal{D}_{\ii}$ (i.e. an intersection of the wires $p$ and $q$) represents a letter $j$ in $\ii$. If the vertex represents the letter $j\in [n-1]$, then $j-1$ is equal to the number of wires running below this intersection. The word $\ii$ can be read off from $\mathcal{D}_{\ii}$ by reading the levels of the vertices from left to right.

For every $\ii\in \W$ the vertices of $\mathcal{D}_{\ii}$ are in bijection with $\Phi^+$ using \eqref{eq:rootident}. The ordering $\le_{\ii}$ defined in \eqref{eq:rootorder} can be read off from $\mathcal{D}_{\ii}$ by reading the vertices from left to right.

\begin{ex} Let $n=5$ and $\ii=(2,1,2,3,4,3,2,1,3,2)$. The wiring diagram $\mathcal{D}_{\ii}$ and the corresponding total ordering $$(2,3)<_{\ii}(1,3)<_{\ii}(1,2)<_{\ii}(1,4)<_{\ii}(1,5)<_{\ii}(4,5)<_{\ii}(2,5)<_{\ii}(3,5)<_{\ii}(2,4)$$ on $\Phi^+$ are depicted below. 
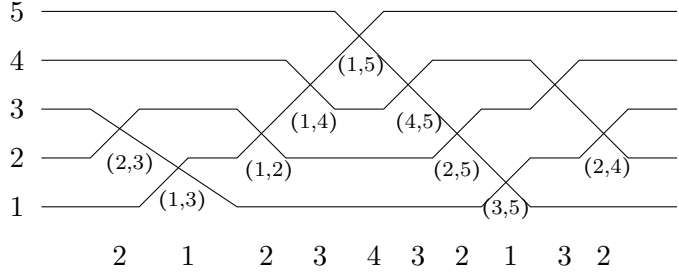
\begin{center}

\begin{tikzpicture}[scale=.65]

\node at (-.5,0) {$1$};
\node at (-.5,1) {$2$};
\node at (-.5,2) {$3$};
\node at (-.5,3) {$4$};
\node at (-.5,4) {$5$};
\node at(1.6,-1){$2$};
\node at(3,-1){$1$};
\node at(4.6,-1){$2$};
\node at(5.7,-1){$3$};
\node at(6.8,-1){$4$};
\node at(7.7,-1){$3$};
\node at(8.6,-1){$2$};
\node at(9.6,-1){$1$};
\node at(10.7,-1){$3$};
\node at(11.5,-1){$2$};

\draw (0,0) --(1,0) -- node[xshift=.9cm,yshift=-.19cm,above]{$\scriptstyle{(1,3)}$} node[xshift=5.2cm,yshift=-0.3cm,above]{$\scriptstyle{(3,5)}$} (2,0) -- (3,1) -- (4,1) -- (5,2) -- (6,3) --(7,4) -- (10,4) -- (13,4);
\draw (0,1) -- (1,1) -- (2,2) -- (4,2) -- (5,1) -- (8,1) -- (9,2) -- (10,2) --(11,3) -- (13,3);
\draw (0,2) -- node[xshift=0.85cm,yshift=-1cm,above]{$\scriptstyle{(2,3)}$} node[xshift=2.7cm,yshift=-1.1cm,above]{$\scriptstyle{(1,2)}$} node[xshift=5.2cm,yshift=-1.1cm,above]{$\scriptstyle{(2,5)}$} (1,2) -- (4,0) -- (9,0) -- (10,1) -- (11,1) -- (12,2) -- (13,2) ;
\draw (0,3) -- node[xshift=2cm,yshift=-1.1cm,above]{$\scriptstyle{(1,4)}$} node[xshift=3.4cm,yshift=-1.1cm,above]{$\scriptstyle{(4,5)}$} node[xshift=5.9cm,yshift=-1.7cm,above]{$\scriptstyle{(2,4)}$} (5,3) -- (6,2) -- (7,2) -- (8,3) --(10,3) -- (11,2) -- (12,1) -- (13,1) ;
\draw (0,4) -- node[xshift=2.3cm,yshift=-1cm,above]{$\scriptstyle{(1,5)}$} (6,4) -- (7,3) -- (8,2) -- (9,1) -- (10,0) -- (13,0) ;
\end{tikzpicture}
\end{center}
\end{ex}

\begin{defi} To $a\in [n-1]$ we associate an oriented graph $\mathcal{D}_{\ii}(a)$ ($\mathcal{D}_{\ii}^{*}(a)$) as follows. We orient the wires $p$ of $\mathcal{D}_{\ii}$ from left to right (right to left) if $p\le a$, and from right to left (left to right) if $p>a$. An \emph{$a$-(dual) Reineke crossing} $\gamma$ is a sequence $(v_1,\ldots,v_k)$ of vertices of $\mathcal{D}_{\ii}$ which are connected by oriented edges $v_1 \rightarrow v_1 \rightarrow \ldots \rightarrow v_{k}$ in $\mathcal{D}_{\ii}(a)$ ($\mathcal{D}_{\ii}^{*}(a)$) satisfying the following two conditions:
\begin{itemize}
\item $v_1$ is the leftmost (rightmost) vertex of the wire $a$ and $v_k$ is the leftmost (rightmost) vertex of the wire $a+1$,
\item whenever $v_j, v_{j+1}, v_{j+2}$ lie on the same wire $p$ in $\mathcal{D}_{\ii}$ and the vertex $v_{j+1}$ lies on the intersection the wires $p$ and $q$, we have
\begin{align*} p > q & \quad \text{if }q\le a, \\ 
p<q & \quad \text{if }a+1 \le q. 
\end{align*}
\end{itemize}

By \cite{GKS1} the set $\Gamma_a=\Gamma_a(\ii)$ ($\Gamma_a^{*}=\Gamma_a^{*}(\ii)$) $a$-(dual) Reineke crossings carries the structure of an order lattice defining $\preceq$ as follows. Let $\gamma_1,\gamma_2\in \Gamma_a$ ($\gamma_1,\gamma_2\in \Gamma^{*}_a)$. We set $\gamma_1 \preceq \gamma_2$ if all vertices of $\gamma_1$ lie in the region of $\mathcal{D}_{\ii}$ cut out by $\gamma_2$.
\end{defi}

\begin{defi} The set $T_\gamma$ of \emph{turning points} of a (dual) $a$-Reineke crossing $\gamma$ consists of the vertices of $\gamma$, such that the oriented edge with sink $\gamma$ and the oriented edge with source $\gamma$ lie on different wires.
\end{defi}

We identify $k\in[N]$ with the $k$-th vertex $(p,q)$ in  $\mathcal{D}_{\ii}$ from left.
\begin{defi}\label{def:vectors}Let $r$ and $s$ be the maps associating to $\gamma$ in either $\Gamma_a$ or $\Gamma^{*}_a$ a vector in $\mathbb{Z}^{N}$ by 
\begin{align*}
\left(r({\gamma})\right)_{p,q} &:= \begin{cases} \sgn(q-p), &  \text{if $v_{p,q}\in T_{\gamma},$} \\
0, & \text{else,}\end{cases} \\
\left(s({\gamma})\right)_{p,q} &:= \begin{cases} 1, & \text{if $v_{p,q}\in \gamma,$ $p\le a <q \text{ or }q\le a <p,$} \\
-1, & v_{p,q}\in \gamma\setminus T_{\gamma}, \ a<p,q \text{ or } p,q\le a,
\\ 0 & \text{else.}\end{cases}
\end{align*}

\end{defi}

\begin{ex}
Let $n=5$. The vertices lying on the red path below form the $3-$rigorous path $\gamma=(v_{3,2},v_{3,1},v_{1,2},v_{2,5},v_{2,4},v_{4,5},v_{4,1})$.

\begin{center}

\begin{tikzpicture}[scale=.75]

\node at (-.5,0) {$1$};
\node at (-.5,1) {$2$};
\node at (-.5,2) {$3$};
\node at (-.5,3) {$4$};
\node at (-.5,4) {$5$};

\draw (0,0) -- node[yshift=-0.033cm]{\textbf{>}} (1,0) -- (2,0) -- (3,1) -- node[yshift=-0.033cm]{\textbf{>}} (4,1) -- (5,2) -- (6,3) --(7,4) -- (10,4) -- node[xshift=-0.8cm,yshift=-0.033cm]{\textbf{>}} (13,4);
\draw (0,1) -- node[yshift=-0.033cm]{\textbf{>}} (1,1) -- (2,2) -- node[yshift=-0.033cm]{\textbf{>}} (4,2) -- (5,1) -- node[yshift=-0.033cm]{\textbf{>}} (8,1) -- (9,2) -- node[yshift=-0.033cm]{\textbf{>}} (10,2) --(11,3) -- node[yshift=-0.033cm]{\textbf{>}} (13,3);
\draw[line width=1.25mm, red] (0,2) -- (1,2)-- (2.8,0.7) -- (3,1) --  (4,1) -- (4.5,1.5) -- (5,1) -- (8,1) -- (9,2) -- (10,2) -- (10.5,2.5) -- (10,3) -- (8,3) -- (7,2) -- (6,2) -- (5,3) -- (0,3); 
\draw (2.8,0.7) -- (4,0) -- node[yshift=-0.033cm]{\textbf{>}} (9,0) -- (10,1) -- node[yshift=-0.033cm]{\textbf{>}} (11,1) -- (12,2) -- (13,2) ;
\draw (0,3) -- node[yshift=-0.033cm]{\textbf{<}} (5,3) -- (6,2) -- node[yshift=-0.033cm]{\textbf{<}} (7,2) -- (8,3) -- node[yshift=-0.033cm]{\textbf{<}} (10,3) -- (11,2) -- (12,1) -- (13,1) ;
\draw (0,4) -- node[xshift=-0.3cm,yshift=-0.033cm]{\textbf{<}} (6,4) -- (7,3) -- (8,2) -- (9,1) -- (10,0) -- node[yshift=-0.033cm]{\textbf{<}} (13,0) ;
\end{tikzpicture}

\end{center}
 We have $$r({\gamma})=(0,-1,1,0,0,0,0,0,1,0), \quad s({\gamma})=(-1,0,0,1,0,-1,1,0,1,0).$$
\end{ex}

\subsection{Order lattice of Kashiwara a-crossings}
We fix $\ii \in \W$. Let $\Phi^+_{\ii}=(\beta_1,\beta_2,\ldots,\beta_N)$ be the sequence of positive roots ordered with respect to $\le_{\ii}$.

For $k\in [N]$ we define the sequence 
$$\upsilon(k)=(\beta_k, \beta_{k+1},\ldots, \beta_{N}).$$
Thus $\upsilon(k)_{\ell}=\beta_{\ell+k-1}$ for $1\le \ell \le N-k-1$.
We define for $a\in [n-1]$ the set of \emph{Kashiwara $a$-crossings} to be $\Upsilon_a(\ii)=\Upsilon_a:=\{\upsilon(k) \mid i_k=a\}$. We further define a total order $\preceq$ on $\Upsilon_a(\ii)$ by $\upsilon(k_1) \preceq \upsilon(k_2) \iff k_2\le k_1.$

\begin{defi} We associate two vectors $r(\upsilon),s(\upsilon) \in\mathbb{Z}^{N}$ to a Kashiwara $a$-crossing $\upsilon$ by 
\begin{equation*}
\left(r(\upsilon)\right)_{\ell} := \begin{cases} 1 & \upsilon_1=\beta_k, \\
0 & \text{else,}\end{cases} \qquad 
\left(s(\upsilon)\right)_{\ell} := \begin{cases} 1 & \upsilon_1=\beta_j, \\
-1 & \upsilon_i=\beta_{\ell} \text{ for }i \ge 2 \text{ and} \ i_{\ell}=a,  \\
2 & \upsilon_i=\beta_{\ell} \text{ for }i \ge 2 \text{ and} \ i_{\ell}\in \{a\pm 1\}
\\ 0 & \text{else.}\end{cases}
\end{equation*}

\end{defi}

\section{Explicit description of crystal structures}
The order lattices $\Gamma_a$, $\Gamma^{*}_a$ and $\Upsilon_a$ with associated vectors can be used to give an explicit description of the crystals as follows. We denote by $\Gamma_a^{\text{opp}}$ the lattice $\Gamma_a$ with reversed order. For $x$ in either of $ \mathcal{L}_{\ii}(\lambda)$, $\mathcal{L}_{\ii}(\lambda)^*,$ $ \mathcal{S}_{\ii}(\lambda)^*$ or  $\mathcal{S}_{\ii}(\lambda)^*$ we set
\begin{align*}
\mathbb{L}_a&:=\Gamma_a,\,&\rho&:=r:\Gamma_a\rightarrow \Z^N ,\quad &\sigma&:=s:\Gamma_a\rightarrow \Z^N, \qquad &&\text{for } x\in\mathcal \mathcal{L}_{\ii}(\lambda),\\
\mathbb{L}_a&:=\Gamma_a^*,\, &\rho&:=r:\Gamma_a^*\rightarrow \Z^N,\quad &\sigma&:=s:\Gamma_a^*\rightarrow \Z^N,\qquad &&\text{for } x\in\mathcal \mathcal{L}_{\ii}(\lambda)^*,\\
\mathbb{L}_a&:=\Gamma_a^{\text{opp}},\, &\rho&:=s:\Gamma_a^{\text{opp}}\rightarrow \Z^N,\quad &\sigma&:=r:\Gamma_a^{\text{opp}}\rightarrow \Z^N,\qquad &&\text{for } x\in\mathcal \mathcal{S}_{\ii}(\lambda)^*,\\
\mathbb{L}_a&:=\Upsilon_a,\, &\rho&:=r:\Upsilon_a\rightarrow \Z^N,\quad &\sigma&:=s:\Upsilon_a\rightarrow \Z^N,\qquad &&\text{for } x\in\mathcal \mathcal{S}_{\ii}(\lambda).
\end{align*}
\begin{thm}\label{crossingformulas} Let $\lambda\in P^+$, $\ii\in \W$ and $a\in [n-1]$. Then for $x$ in either of $ \mathcal{L}_{\ii}(\lambda)$, $\mathcal{L}_{\ii}(\lambda)^*,$ $ \mathcal{S}_{\ii}(\lambda)^*$ or  $\mathcal{S}_{\ii}(\lambda)^*$ we have
\begin{align*}
\varepsilon_a (x) &= \max\{\left\langle x,\sigma(\gamma)\right\rangle \mid \gamma\in \mathbb{L}_a\},\\
{f}_a(x) &= \begin{cases} x+\rho(\gamma_{x}) & \text{ if }\varphi_a(x)>0,\\
0 & \text{ else,} \end{cases}\\
{e}_a(x) &= \begin{cases} x-\rho(\gamma^{x}) & \text{ if }\varepsilon_a(x)>0,\\
0 & \text{ else,} \end{cases}
\end{align*}
where $\gamma^x\in\mathbb{L}_a$ is minimal with $\langle x, \sigma(\gamma^x)\rangle=\varepsilon_a(x)$ and $\gamma_x \in\mathbb{L}_a$ is maximal with $\langle x, \sigma(\gamma_x)\rangle=\varepsilon_a(x)$.
\end{thm}
\begin{proof}
For $x\in\mathcal{L}_{\ii}(\lambda)$ and $x\in\mathcal{L}_{\ii}(\lambda)^*$ the claim is a consequence of \cite[Theorem 2.13]{GKS1} and \cite[Theorem 2.20]{GKS1}, respectively, using the embeddings $B(\lambda)\hookrightarrow B(\infty)$ and $B(\lambda) \hookrightarrow B(\infty)^*$ given by \eqref{lambdaininfty} and \eqref{lambdaininftystar}. For $x\in\mathcal{S}_{\ii}(\lambda)^*$ and $x\in\mathcal{L}_{\ii}(\lambda)^*$ the claim is proved in \cite{GKS3} and for $x\in\mathcal{S}_{\ii}(\lambda)$ it is a reformulation of the crystal structure given in Section \ref{stringcrys}.
\end{proof}

\section{Associated cones and polytopes}\label{conesandpol}

In this section we derive defining inequalities of the cones and polytopes of Lusztig- and string data from the order lattices defined in Section \ref{orderlattices}.

\subsection{Definining inequalities}\label{sec:inequ}

\begin{thm}\label{2}
Let $\mathcal{L}_{\ii}(\infty)_{\mathbb{R}}$ and $\mathcal{S}_{\ii}(\infty)_{\mathbb{R}}$ be the conic hulls in $\mathbb{R}^N$ of $\mathcal{L}_{\ii}(\infty)$ and $\mathcal{S}_{\ii}(\infty)$, respectively. We have 
\begin{align} \label{eq:a}
\mathcal{L}_{\ii}(\infty)_{\mathbb{R}}&=\{ x \in \mathbb{R}^N \mid \left<r(\upsilon),x\right> \ge 0 \ \forall a\in[n],\, \forall\upsilon\in \Upsilon_a\} \\ \label{eq:b}
\mathcal{S}_{\ii}(\infty)_{\mathbb{R}}&=\{ x \in \mathbb{R}^N \mid \left<r(\gamma),x\right> \ge 0 \  \forall a\in[n],\, \forall\gamma\in \Gamma^{*}_a\}.
\end{align}

For $\lambda\in P^+$ let $\mathcal{L}_{\ii}(\lambda)_{\mathbb{R}}$, $\mathcal{L}_{\ii}(\lambda)^*_{\mathbb{R}}$, $\mathcal{S}_{\ii}(\lambda)^{*}_{\mathbb{R}}$ and $\mathcal{S}_{\ii}(\lambda)_{\mathbb{R}}$  be the subsets of vectors in $\mathbb{R}^N$ satisfying the defining inequalities of $\mathcal{L}_{\ii}(\lambda)$, $\mathcal{L}_{\ii}(\lambda)^*$, $\mathcal{S}_{\ii}(\lambda)^{*}$ and $\mathcal{S}_{\ii}(\lambda)$, respectively. We have \begin{align*}
\mathcal{L}_{\ii}(\lambda)_{\mathbb{R}}&=\{x\in \mathcal{L}_{\ii}(\infty)_{\mathbb{R}} \mid \lambda_a - \left<s(\gamma),x\right> \ge 0 \  \forall a\in[n],\, \forall \gamma\in \Gamma^{*}_a\}, \\
\mathcal{L}_{\ii}(\lambda)^*_{\mathbb{R}}&=\{x\in  \mathcal{L}_{\ii}(\infty)_{\mathbb{R}} \mid \lambda_a - \left<s(\gamma),x\right> \ge 0 \  \forall a\in[n],\, \forall \gamma\in \Gamma_a\}, \\ 
\mathcal{S}_{\ii}(\lambda)^*_{\mathbb{R}}&=\{x \in \mathcal{S}_{\ii}(\infty)_{\mathbb{R}} \mid \lambda_a-\left<r(\gamma),x\right> \ge 0 \ \forall a\in[n], \forall \gamma \in \Gamma_a\},\\
\mathcal{S}_{\ii}(\lambda)_{\mathbb{R}}&=\{x \in \mathcal{S}_{\ii}(\infty)_{\mathbb{R}} \mid \lambda_a-\left<s(\upsilon),x\right> \ge 0 \ \forall a\in[n], \forall \upsilon \in \Upsilon_a \}.
\end{align*}
\end{thm}

\begin{proof} Equality \eqref{eq:a} follows from the definition. Equality \eqref{eq:b} was proven in \cite{GP} (see also \cite[Theorem 4.4]{GKS1}). 

The defining inequalities for the polytopes follow from \eqref{lambdaininfty}, \eqref{lambdaininftystar} and Theorem \ref{crossingformulas}.
\end{proof}

\subsection{Unimodular isomorphisms between polytopes}

For $a\in [n-1]$ we set $a^*=N-i$. We fix $\lambda=\sum_{a=1}^{n-1}\lambda_a \omega_a\in P^+$ and write $\underline{\lambda}:=(\lambda_{i_1},\lambda_{i_2},\ldots \lambda_{i_N})\in \mathbb{N}^N$ and $\lambda^*:=\sum_{a=1}^{n-1}\lambda_{a^*}\omega_{a}$ . We fix $\ii=(i_1,\ldots,i_N)\in \W$ and set $\ii^*=(i_1^*,\ldots, i_N^*)$.

We call two polytopes $P_1,P_2 \subset \mathbb{R}^m$ \emph{affine unimodular isomorphic} if there exists a lattice isomorphism $g: \mathbb{Z}^m \rightarrow \mathbb{Z}^m$ and a vector $v\in \mathbb{Z}^m$ such that $g(P_1)+v=P_2$.

Let $(c_{i,j})$ be the Cartan matrix of $\text{sl}_n(\mathbb{C})$. We define the maps
\begin{align}
\notag
F_{\ii}(\lambda): \mathbb{Z}^N &\rightarrow\mathbb{Z}^N,\qquad (F_{\ii}(x))_k=x_k+\displaystyle\sum_{N\ge \ell > k} c_{i_k,i_{\ell}}x_{\ell}, \\
\notag
G_{\ii}(\lambda): \mathbb{Z}^N &\rightarrow\mathbb{Z}^N,\qquad
G_{\ii}(\lambda)(x)=\underline{\lambda}-F_{\ii}(x),\\
\opp:\mathbb{Z}^N &\rightarrow\mathbb{Z}^N,\qquad\notag
(x_1,\dots, x_N)^{\opp}=(x_N, \dots, x_1).
\end{align}

\begin{prop}\label{unimod} Let $\lambda \in P^+$ and $\ii \in \W$. We have affine unimodular isomorphisms:

$$\mathcal{S}_{\ii}(\lambda)_{\mathbb{R}}^* \xrightarrow{G_{\ii}(\lambda)} \mathcal{L}_{\ii}(\lambda^*)_{\mathbb{R}} \xleftarrow{\opp} \mathcal{L}_{\ii^*}(\lambda)_{\mathbb{R}}^*.$$

\end{prop}
\begin{proof}
The fact that $G_{\ii}(\lambda)$ is an affine unimodular isomorphism is proved in Proposition 8.3 of \cite{GKS2} and can also be deduced from \cite{CMM}.

Note that $*$ on $\mathcal{L}_{\ii}(\infty)$ induces a bijection between $\mathcal{L}_{\ii}(\lambda)$ and $\mathcal{L}_{\ii}(\lambda)^*$. By \cite{L93} we have for $x\in \mathcal{L}_{\ii}$, that $\Phi^{\ii}_{\ii^*}(x^*)=x^{\text{opp}}$ which proves that $\text{opp}$ is a unimodular isomorphism between $\mathcal{L}_{\ii}(\lambda^*)$ and $\mathcal{L}_{\ii^*}(\lambda)^*$.
\end{proof}

\begin{thm}\label{new} 
\begin{enumerate}
\item The map $G_{\ii}(\lambda)$ induces a bijection between the defining inequalities of $\mathcal{S}_{\ii}(\lambda)^*_{\mathbb{R}}$ and the defining inequalities of $\mathcal{L}_{\ii}(\lambda^*)_{\mathbb{R}}$ given in Theorem \ref{2}. 
\item The map $\opp$ induces a bijection between the defining inequalities of $\mathcal{L}_{\ii^*}(\lambda)^*_{\mathbb{R}}$ and  the defining inequalities of $\mathcal{L}_{\ii}(\lambda^*)_{\mathbb{R}}$ given in Theorem \ref{2}.
\end{enumerate}

\end{thm}
\begin{proof}
Let $x\in \mathcal{L}_{\ii}(\lambda^*)$ and $\gamma\in \Gamma_{a^*}(\ii)$. By the definition of $G_{\ii}$ we have
$$\lambda_{a^*}-\left\langle s(\gamma), G_{\ii}(\lambda)(x)\right\rangle=\lambda_{a^*}-\left\langle s(\gamma),\underline{\lambda}\right\rangle +\left\langle F_{\ii}^{t}\circ s(\gamma), x\right\rangle.$$
Now
$$\left<s(\gamma),\underline{\lambda}\right>=\displaystyle\sum_{b^*\in [n-1]}\sum_{k: \ i_k=b^*}\lambda_{i_k}s(\gamma)_k.$$
By \cite[Lemma 9.5]{GKS3} we have $\sum_{k: \ i_k=b^*}\lambda_{i_k}s(\gamma)_k=\delta_{b^*,a^*}$. Hence $\left\langle s(\gamma),\underline{\lambda}\right\rangle=\lambda_{a^*}$.
Note that $F_{\ii}^{t}(x)=x_k+\sum_{1\le \ell < k}c_{i_k,i_{\ell}}x_{\ell}$. A direct calculation shows that $F_{\ii}^{t}=[\text{CA}_r]_{trop}\circ[\text{NA}_r^{-1}]_{trop}$ where $[\text{CA}_r]_{trop}$ and $[\text{NA}_r]_{trop}$ are the tropicalisation of the maps defined in Equation (92) and (93) of \cite{G}. Theorem 3.11 of op. cit. then shows $F_{\ii}^{t}\circ s(\gamma)=r(\gamma)$. Conclusively $$\lambda_{a^*}-\left<s(\gamma), G_{\ii}(\lambda)(x)\right>=\left<r(\gamma),x\right>.$$

Fix $k\in N$ and let $a\in [n-1]$ be such that $i_k=a$. We have $$\left\langle r(\upsilon(k)), G_{\ii}(x)\right\rangle =\lambda_a-(F_{\ii}(x))_k=\lambda_a-\left\langle s(\upsilon(k)),x\right\rangle.$$ Hence we have proved the first statement of the Theorem.

The second statement is immediate from the constructions.
\end{proof}

\begin{rem} One may easily check that $\opp:\mathcal{L}_{\ii^*}(\lambda)^* \rightarrow \mathcal{L}_{\ii}(\lambda^*)$ induces an isomorphism of crystals. By \cite{GKS2} (see also \cite{GKS3}) we have $\Phi^{\jj}_{\ii} \circ G_{\jj}(\lambda) \circ \Psi^{\ii}_{\jj}=G_{\ii}(\lambda)$ and thus $G_{\ii}(\lambda):\mathcal{S}_{\ii}(\lambda)^* \rightarrow \mathcal{L}_{\ii}(\lambda^*)$ induces an anti-isomorphism of crystals (i.e. ${f}_a(G_{\ii}(\lambda)(x))=G_{\ii}(\lambda)({e}_a x)$ for $a \in [n-1]$). 

For the lexicographical minimal reduced word there furthermore are linear crystal isomorphism between each pair of polytopes appearing in Theorem \ref{2} leading to piecewise linear bijection for all reduced words. We elaborate on this in \cite{GKS3}.
\end{rem}

\def\cprime{$'$} \def\cprime{$'$} \def\cprime{$'$} \def\cprime{$'$}

\end{document}